\documentclass[12pt]{amsart} 
 
\usepackage{amssymb}

\numberwithin{equation}{section} 
 
\newtheorem{theorem}{Theorem}[section] 
\newtheorem{proposition}[theorem]{Proposition} 
\newtheorem{lemma}[theorem]{Lemma}
\newtheorem{definition}[theorem]{Definition}

\begin{document} 
 
\title[Solvable subgroups of semisimple algebraic groups]{On the maximal
solvable subgroups of semisimple algebraic groups}
 
\author[H. Azad]{Hassan Azad}

\address{Department of Mathematics and Statistics,
King Fahd University of Petroleum and Minerals,
Dhahran 31261, Saudi Arabia}

\email{hassanaz@kfupm.edu.sa}

\author[I. Biswas]{Indranil Biswas} 
 
\address{School of Mathematics, Tata Institute of 
Fundamental Research, Homi Bhabha Road, Mumbai 400005, India} 
 
\email{indranil@math.tifr.res.in} 
 
\author[P. Chatterjee]{Pralay Chatterjee} 
 
\address{The Institute of Mathematical Sciences, C.I.T. Campus, 
Taramani, Chennai 600113, India} 

\email{pralay@imsc.res.in} 
 
 
 
\begin{abstract} 
Let $G$ be a semisimple affine algebraic group defined over a field $k$ of 
characteristic zero. We describe all the maximal connected solvable subgroups
of $G$, defined over $k$, up to conjugation by rational points of $G$.
\end{abstract}
 
\maketitle 

\section{Introduction}

Let $G$ be a semisimple affine algebraic group defined over a field $k$ of 
characteristic zero (the field need not be algebraically closed).
The group of $k$--rational points of $G$ will be denoted by $G(k)$. Our aim 
here is to describe all the maximal connected solvable $k$--subgroups of $G$ 
up to conjugation by elements in $G(k)$ in terms of certain solvable 
$k$--subgroups of some of the standard parabolic $k$--subgroups containing 
a fixed minimal $k$--parabolic subgroup. 

Similar works have been done earlier considering different set-ups. When 
$k\,=\, {\mathbb R}$, the analogous problem for real semisimple Lie algebras 
and real semisimple algebraic groups were studied in \cite[Theorem 4.1]{Mo} and 
\cite[Section 3]{Ma} respectively. It was proved by Platonov, \cite{P}, that 
the number of conjugacy classes of maximal solvable subgroups (not necessarily 
connected) in an algebraic group over an algebraically closed field is finite. 

First assume that $G$ is $k$--anisotropic. Then the group of $k$--rational 
points $G(k)$ has no unipotent elements. Therefore, the maximal connected 
solvable $k$--subgroups of $G$ are
precisely the maximal tori defined over $k$; these tori are all 
$k$--anisotropic. Thus, in this case the maximal connected solvable 
$k$--subgroups of $G$ are precisely the maximal $k$--anisotropic tori of $G$.
Although the problem of finding $G(k)$--conjugacy classes of 
maximal $k$--tori in $k$--anisotropic groups may be tractable for special
cases of $k$, resolving the problem for a general $k$ seems very difficult to 
the best of our knowledge. In what follows we assume that $G$ is 
$k$--isotropic.

The reader is referred to Section \ref{sec-notation}
for the definitions and notation used here.
Fix a maximal $k$--split torus $S$ of $G$. Let $\Delta$ be the set of 
$k$--roots with respect to $S$, and let $\Delta^+\, \subset\, \Delta$ be the 
positive roots given by a fixed minimal $k$--parabolic subgroup of $G$ 
containing $S$. Let $\Phi\, \subset\, \Delta^+$ be the subset consisting 
the simple roots. Any subset $\Theta$ of $\Phi$ defines a $k$--parabolic subgroup 
$P_\Theta$ of $G$. Define $S_\Theta$ as in \eqref{f8}. Let $Z_G(S_\Theta)$ be 
the centralizer of $S_\Theta$ in $G$.

We need to make a definition for the convenience of exposition. A subset 
$\Theta\, \subset\, \Phi$ is called {\it admissible} if $[Z_G(S_\Theta)\, 
,Z_G(S_\Theta)]$ admits a maximal $k$--torus which is $k$--anisotropic.
Let $\Theta\, \subset\, \Phi$ be an admissible subset, and let $T$ 
be a maximal $k$--torus of $[Z_G(S_\Theta)\, ,Z_G(S_\Theta)]$ which is 
$k$--anisotropic. It is clear that $TZ(Z_G(S_\Theta)) R_u(P_\Theta)$ is a 
connected solvable $k$--subgroup of $G$. 

We prove the following theorem (see Theorem \ref{thm1}, Theorem \ref{thm2}
and Proposition \ref{thm3}):

\begin{theorem}\label{thm0}
Let $\Theta\, \subset\, \Phi$ be an admissible subset, and let $T$ be 
a maximal $k$-torus of $[Z_G(S_\Theta)\, ,Z_G(S_\Theta)]$ which is $k$--anisotropic.
Then the subgroup $$B_{\Theta,T}\, :=\, TZ(Z_G(S_\Theta)) R_u(P_\Theta)$$ is a 
maximal connected solvable $k$--subgroup of $G$.

For any maximal connected solvable $k$--subgroup $B$ of $G$, there is an 
admissible subset $\Theta$ and a maximal anisotropic $k$-torus $T\, \subset\, 
[Z_G(S_\Theta)\, ,Z_G(S_\Theta)]$ such that $B$ 
is conjugate to $B_{\Theta,T}$ (defined above) by some element in $G(k)$.

For admissible subsets $\Theta_i$, $i\,=\, 1\, ,2$, and maximal 
$k$--torus $T_i$ of $[Z_G(S_{\Theta_i})\, ,Z_G(S_{\Theta_i})]$ which is
$k$--anisotropic, the two subgroups 
$B_{\Theta_1, T_1}$ and $B_{\Theta_2, T_2}$ of $G$ are conjugate 
by some element in $G(k)$ if and only if $$\Theta_1\,= \,\Theta_2
~\,~ \text{ and }~\,~ c \,T_1 \, c^{-1} \,= \,T_2$$ for some $c\,\in\,
P_{\Theta_1}(k)\,=\,P_{\Theta_2} (k)$ satisfying the condition that $c \, 
Z_G (S_{\Theta_1}) \, c^{-1} \,=\, Z_G (S_{\Theta_1})$.
\end{theorem}

We also give a criterion for an element of $G(k)$ to lie in some maximal 
connected solvable $k$--subgroup of $G$ (see Theorem \ref{C}).

\section{Notation and preliminaries}\label{sec-notation}

In this section we fix some notation, which will be used throughout. For the 
generalities in the theory of algebraic groups that are used here, 
the reader is referred to \cite{BT2} and \cite[Chapter V]{Bo}.
As before, $k$ is a field of characteristic zero, which is not necessarily 
algebraically closed.

The {\it center} of a group $H$ is denoted by $Z(H)$. Let $H$ be a linear 
algebraic group defined over $k$. We denote its Lie algebra by $\text{Lie}(H)$. 
The connected component of $H$, containing the identity element, is denoted by 
$H^0$. For a subgroup $J$ of $H$, and a subset $S$ of $H$, by $Z_{J} (S)$ we 
will denote the subgroup of $J$ that commutes with all the elements of $S$. The 
{\it normalizer} of $J$ in $H$ is denoted by $N_H (J)$.

Let $G$ be a semisimple algebraic group defined over $k$. If
$G$ admits a $k$-split torus of positive dimension, then
$G$ is said to be $k$--\textit{isotropic}; otherwise, $G$ is called 
$k$--\textit{anisotropic}. 

Let $S\, \subset\, G$ be a maximal $k$--split torus. The group of characters
of $S$ will be denoted by $X(S)$.

We fix some notation:
\begin{itemize}
\item $P$ is a fixed minimal $k$--parabolic subgroup of $G$ containing $S$. 

\item $\Delta\, \subset \, X(S)$
is the set of $k$--roots with respect to $S$.

\item $\Delta^+\, \subset\, \Delta$
is the set of positive roots given by $P$.

\item $\Phi\, \subset\, \Delta^+$ is the subset consisting of simple roots 
of $\Delta^+$.
\end{itemize}

For any $\Theta\, \subset\, \Phi$, define
\begin{equation}\label{f8}
S_\Theta\, :=\, (\bigcap_{\chi\in \Theta} \text{kernel}(\chi))^0\, .
\end{equation}
This $k$--split torus $S_\Theta$ is called the \textit{standard $k$--split 
torus of type $\Theta$}.
Let $Z_G(S_\Theta)$ denote the centralizer of $S_\Theta$ in $G$.
The standard $k$--parabolic subgroup of $G$, containing $P$, corresponding to 
$\Theta$ will be 
denoted by $P_\Theta$ (see \cite[p. 197, Section 14.17]{Bo} when $k = \bar{k}$ and
\cite[p. 233, Section 21.11]{Bo} for a general $k$). We recall that
$$
P_\Theta\,=\, Z_G(S_\Theta)\cdot R_u(P_\Theta)\, ,
$$
where $R_u(P_\Theta)$ is the unipotent radical of $P_\Theta$.

It is known that
\begin{equation}\label{A}
Z(Z_G(S_\Theta))^0 \,=\, A\cdot S_\Theta\, ,
\end{equation}
where $A$ is 
a $k$--anisotropic torus; see \cite[Proposition 1.1]{DT}. 
Therefore, the $k$--split part of
$Z(Z_G(S_\Theta))^0$ is $S_\Theta$. In particular, $Z_G(S_\Theta)/S_\Theta$ admits
a $k$--anisotropic maximal torus if and only if $[Z_G(S_\Theta)\, ,
Z_G(S_\Theta)]$ admits a $k$--anisotropic maximal torus.

\begin{definition}
{\rm A subset $\Theta\, \subset\, \Phi$ is called} admissible {\rm if
$[Z_G(S_\Theta)\, ,Z_G(S_\Theta)]$ admits a maximal $k$--torus which is 
$k$--anisotropic. In the case when $k = {\mathbb R}$ this is equivalent to the
definition of admissible subsets of $\Phi$ given in \cite[Definition 5.8]{Ch}.}
\end{definition}

\section{A collection of maximal connected solvable subgroups}\label{sec1}

\begin{lemma}\label{lem1} Let $G$ be a semisimple algebraic group defined over $k$.
Suppose that $G$ admits a maximal $k$--torus, say $T$, which is 
$k$--anisotropic. Then there is no nontrivial unipotent $k$-subgroup $U\, 
\subset\, G$ such that $T\, \subset\, N_G(U)$.
\end{lemma}

\begin{proof}
If $G$ is $k$--anisotropic, then there no nontrivial unipotent $k$--subgroup of
$G$. Hence we will assume that $G$ is $k$--isotropic.

As before, let $S$ be a maximal $k$--split torus in $G$. To prove the lemma 
by contradiction, let $U\, \not=\, \{e\}$ be a unipotent $k$--subgroup 
so that $T\, \subset\, N_G(U)$. Using \cite[Proposition 3.1]{BT} we see that there is 
a parabolic $k$--subgroup $P\, \subset\, G$ such that
$$
N_G(U)\, \subset\, ~ P \,~ \text{ and } ~\, ~U \subset R_u(P)\, .
$$

Now, there is a subset $\Theta\,\subset\, \Phi$ such that $P$ is conjugate to 
$P_\Theta$ by some element in $G(k)$. Fix $\alpha\, \in\, G(k)$ such that 
$\alpha P \alpha^{-1}\, =\, P_\Theta$. Note that $P_\Theta\, \subsetneq\, G$
because $U\, \subset\, R_u(P)$ and $U\,\not=\,\{e\}$. Clearly, we have
$$
\alpha T \alpha^{-1}\, \subset\, P_\Theta\, .
$$
As $Z_G(S_\Theta)$ is a maximal reductive subgroup of $P_\Theta$ defined over $k$,
it follows that there is an element $\beta\, \in\, G(k)$ such that
$$
\beta T \beta^{-1}\, \subset\, Z_G(S_\Theta)\, .
$$
Define $T'\,:=\, \beta T\beta^{-1}$. This $T'$ is a maximal torus, and it is
$k$--anisotropic; also, $T'$ commutes with $S_\Theta$. Therefore,
$$
S_\Theta\, \subset\, Z_G(T')\,=\, T'\, .
$$
But this is in contradiction with the facts that $S_\Theta$ is positive 
dimensional and $k$--split while $T'$ is $k$--anisotropic. In view of this 
contradiction, the proof of the lemma is complete.
\end{proof}

In the next lemma we will deal with a semisimple group $H$ over the algebraic 
closure $\overline{k}$ of $k$. As in the case of $k$, we have a 
description of all the parabolic subgroups containing
a fixed Borel subgroup (see \cite[ p. 197, Section 14.17]{Bo}).

\begin{lemma}\label{lem2}
Let $H$ be a semisimple algebraic group defined over $\overline{k}$. Let
$P\, \subset\, H$ be a parabolic subgroup, and let $D \,\subset\, H$ be a 
connected solvable subgroup of $H$.
Let $T$ be a maximal torus of $H$ so that $T \,\subset\, D \cap P$.
Further assume that $R_u (P)\,\subset\, R_u (D)$. Then, $D\,\subset\, P$. 
\end{lemma}

\begin{proof}
Since both $D$ and $P$ are connected, it is enough to show that
\begin{equation}\label{zh1}
\text{Lie}(D) \,\subset\, \text{Lie}(P)\, .
\end{equation}

To prove \eqref{zh1} by contradiction, suppose 
$\text{Lie}(D)$ is not contained in $\text{Lie}(P)$. We fix a Borel subgroup $B 
\,\subset\,P$ containing $T$. Let $\widetilde{\Delta}$ be the set of roots with 
respect to $T$. Let $\widetilde{\Delta}^+$ be the set of positive roots induced 
by $B$, and let $\widetilde{\Phi}$ be the set of simple roots in 
$\widetilde{\Delta}^+$. Then there is a subset $\Theta\,\subset\, 
\widetilde{\Phi}$ such that $P\,=\,P_\Theta$.

Denote the ${\mathbb Z}$--span of $\Theta$ by ${\mathbb Z}\cdot\Theta$. As 
$\text{Lie}(D)$ is $T$--invariant under the adjoint action, and $\text{Lie}(D)$ 
is not contained in $\text{Lie}(P)$, we conclude that there is an element 
$\alpha\,\in\,\widetilde{\Delta}^+ - {\mathbb Z}\cdot\Theta$ such that 
$\text{Lie}(H)_{-\alpha}\,\subset\,\text{Lie}(D)$. As $R_u (P)\,\subset \,R_u 
(D)$, it follows that $\text{Lie}(H)_\alpha\,\subset\, \text{Lie}(D)$. Thus
$$
\text{Lie}(H)_{-\alpha} + \text{Lie}(T) + \text{Lie}(H)_\alpha\,\subset 
\,\text{Lie}(D)\, .
$$
But $\text{Lie}(D)$ is a solvable Lie algebra, while $\text{Lie}(H )_{-\alpha} 
+ \text{Lie}(T) + \text{Lie}(H)_\alpha$ contains a copy of ${\mathfrak 
s}{\mathfrak l}_2 (\overline{k})$. This is a contradiction, proving
\eqref{zh1}.
\end{proof}

Let $$\Theta\, \subset\,\Phi$$ be a subset such that $Z_G(S_\Theta)/S_\Theta$ 
admits a $k$--anisotropic maximal torus. Recall that 
this is equivalent to the assertion that $[Z_G(S_\Theta)\, ,Z_G 
(S_\Theta)]$ admits a $k$--anisotropic maximal torus. Let
$$
T \, \subset\, [Z_G(S_\Theta)\, , Z_G(S_\Theta)]
$$
be a $k$--anisotropic maximal torus of $[Z_G(S_\Theta)\, , Z_G(S_\Theta)]$.
Note that $$T \, Z(Z_G(S_\Theta)) \,=\, T \, Z(Z_G(S_\Theta))^0\, .$$
Clearly, $T \,Z(Z_G(S_\Theta))$ is a maximal $k$--torus of $Z_G (S_\Theta)$.

\begin{theorem}\label{thm1}
In the above set--up, $$B_{\Theta,T}\, :=\, TZ(Z_G(S_\Theta)) R_u(P_\Theta)$$ 
is a maximal connected solvable $k$--subgroup of $G$.
\end{theorem}

\begin{proof}
Clearly $B_{\Theta,T}$ is a connected solvable $k$--subgroup.
Let $B\, \subset\, G$ be a connected solvable $k$--subgroup
such that $B_{\Theta,T}\,\subset\, B$. 
We set $T_\Theta \,:=\,T'_\Theta \,Z(Z_G(S_\Theta))$.
Since $T_\Theta$ is a maximal 
torus of $G$, and $T_\Theta\, \subset\, B_{\Theta,T} \,\subset\, B$, 
we conclude that $B\,=\, T_\Theta R_u(B)$. Further, as $B_{\Theta,T}\, 
\subset\, B$, it follows that
$$R_u(B_{\Theta,T})\,=\, R_u(P_\Theta)\,\subset\, R_u(B)\, .
$$

Therefore, to prove the theorem, it suffices to show that
\begin{equation}\label{f1}
R_u(B_{\Theta,T})\,=\, R_u(B)\, .
\end{equation}

As $T_\Theta$ is a maximal torus in $G$ contained in $B$,
and $R_u(P_\Theta)\,\subset\, R_u(B)$, from Lemma \ref{lem2} it follows
that $$B\, \subset\, P_\Theta\, .$$ Since
$R_u(P_\Theta)\, \subset\, R_u(B)\, \subset\, P_\Theta$, and
$Z_G(S_\Theta)\,=\, [Z_G(S_\Theta)\, ,Z_G(S_\Theta)]Z(Z_G(S_\Theta))$,
one has that
$$
R_u(B)\,=\, (Z_G(S_\Theta)\cap R_u(B))R_u(P_\Theta)\,=\,
([Z_G(S_\Theta)\, ,Z_G(S_\Theta)]\cap R_u(B))R_u(P_\Theta)\, .
$$
Clearly, $T_\Theta\, \subset\, N_G(R_u(B))$. Hence
$$
T \, \subset\, N_{[Z_G(S_\Theta) ,Z_G(S_\Theta)]}(
[Z_G(S_\Theta)\, ,Z_G(S_\Theta)]\cap R_u(B))\, .
$$
But recall that $T$ is a maximal torus in $[Z_G(S_\Theta)\, 
,Z_G(S_\Theta)]$, and $T$ is $k$--anisotropic. Therefore, applying 
Lemma \ref{lem1} to the semisimple group $[Z_G(S_\Theta)\, ,Z_G(S_\Theta)]$ 
and its unipotent subgroup $[Z_G(S_\Theta)\, ,Z_G(S_\Theta)]\cap R_u(B)$, we 
conclude that
$$
[Z_G(S_\Theta)\, ,Z_G(S_\Theta)]\cap R_u(B)\,=\, \{e\}\, .
$$
Hence $R_u(B)\,=\, R_u(P_\Theta)\,=\, R_u(B_{\Theta,T})$, proving 
\eqref{f1}.
\end{proof}

\section{Completeness of the collection up to conjugation}

\begin{lemma}\label{lem3}
Let $G$ be a semisimple algebraic group defined over $k$. Let
$A\, \subset\, G$ be a $k$--split torus. Let $T\, \subset\, G$ be 
another $k$-torus containing $A$. Assume that there is 
no unipotent $k$-subgroup $U$ of $G$ such 
that $T\, \subset\, N_G(U)$. Then $A\,=\, \{e\}$.
\end{lemma}

\begin{proof}
It is enough to show that $A\, \subset\, Z(G)$, or, equivalently,
\begin{equation}\label{f2}
\text{Ad}(s)\,=\, \text{Id}_{{\rm Lie}(G)}
\end{equation}
for all $s\, \in\, A$. Let
$$
\Gamma\, \subset\, X(A)
$$
be the finite subset of characters such that
$$
\text{Lie}(G)\,=\, \bigoplus_{\chi \in \Gamma} \text{Lie}(G)_\chi\, 
~\, \text{ and }\, ~ \text{Lie}(G)_\chi\,\not=\, 0\, \,\, \forall\,
\chi\,\in\, \Gamma
$$
($\text{Lie}(G)_\chi\, \subset\, \text{Lie}(G)$ is the weight-space, under the 
adjoint action of $A$,
corresponding to the character $\chi$). Clearly, \eqref{f2} is equivalent to 
the statement that
\begin{equation}\label{f3}
\Gamma\,=\, \{1\}\, ,
\end{equation}
where $1\, \in\, X(A)$ is the trivial character of $A$.

To prove \eqref{f3} using contradiction, assume that $\Gamma\,\not=\, \{1\}$. 
Take any nontrivial character $\chi_0\, \in\, \Gamma$. Define
$$
M\,:=\, \bigoplus_{m>0} \text{Lie}(G)_{m\cdot \chi_0}\, .
$$
Let $\exp$ be the usual exponential map from the $k$--subvariety of nilpotent 
elements in $\text{Lie}(G)$ to the $k$--subvariety of unipotent elements in 
$G$. Define $U\, :=\, \exp(M)\, \subset\, G$, which is a unipotent
$k$--subgroup. We have 
$T\, \subset\, N_G(U)$, because $\text{Ad}(T)(M)\,=\, M$. But $\exp 
(\text{Lie}(G)_{\chi_0}) \, \subset\, U$. In particular, $U\, \not=\, \{e\}$, 
which is in contradiction with the assumption in the lemma. Therefore, we have 
proved that \eqref{f3} holds; hence \eqref{f2} holds.
\end{proof}

\begin{theorem}\label{thm2}
Let $G$ be a semisimple affine algebraic group defined over $k$, and let $B$ be 
a maximal connected solvable $k$--subgroup of $G$. Then there is an admissible 
subset $\Theta\, \subset\, \Phi$, a maximal $k$--torus $T \,\subset
\, [Z_G(S_\Theta)\, ,Z_G(S_\Theta)]$ which is $k$--anisotropic and 
an element $\alpha\, \in\, G(k)$, such that the following holds:
$$
\alpha B \alpha^{-1}\,=\, T Z(Z_G (S_\Theta))R_u 
(P_\Theta)\, :=\, B_{\Theta,T}\, .
$$
\end{theorem}

\begin{proof}
We have $B\, \subset\, N_G(R_u(B))$. Let $T'$ be a maximal $k$--torus of $B$
such that $B\,=\, T' R_u(B)$. Therefore, we have
$$
T' \, \subset\, N_G(R_u(B))\, .
$$
Further, we have
$$
R_u(B)\, \subset\, R_u(N_G(R_u(B))) ~\,~ \text{ and }~\, ~
T' R_u(N_G(R_u(B)))\, \supset\, T' R_u(B)\,=\, B\, .
$$

Since $T' R_u(N_G(R_u(B)))$ is a connected solvable $k$--subgroup,
by maximality of $B$, we have
\begin{equation}\label{e1}
R_u(N_G(R_u(B)))\,=\, R_u(B)\, .
\end{equation}
So, using \cite[Corollaire 3.2]{BT} it follows that $N_G(R_u(B))$ is a 
parabolic $k$--subgroup of $G$. This parabolic
$k$--subgroup $N_G(R_u(B))$ will be denoted by $Q$. From 
\eqref{e1} it follows that $R_u(Q) \,=\, R_u(B)$.

There is an element $\delta\, \in\, G(k)$, and a subset $\Theta\, \subset\,
\Phi$, such that
\begin{equation}\label{g1}
\delta Q\delta^{-1} \,=\, P_\Theta\, .
\end{equation}

It is enough to prove the theorem for the group $\delta^{-1} B\delta$ 
instead of $B$. In the rest of the proof we replace $B$ by $\delta^{-1} 
B\delta$.

With this substitution, \eqref{g1} becomes
$$
N_G(R_u(B))\,=\, Q\,=\, P_\Theta\, .
$$
Consequently,
\begin{equation}\label{f10}
R_u(B)\,=\, R_u(P_\Theta)~ \,~ \text{ and } ~\, ~B \subset P_\Theta\, .
\end{equation}
As $Z_G(S_\Theta)$ is a Levi $k$--subgroup of $P_\Theta$, there 
is a $k$--rational point $\gamma\,\in\, R_u (P_\Theta) (k)$ such that
\begin{equation}\label{f5}
\gamma T' \gamma^{-1} \, \subset\, Z_G(S_\Theta)\, .
\end{equation}

We now substitute the maximal $k$--torus
$$
\widehat{T}\, :=\, \gamma T' \gamma^{-1}\, \subset\, B
$$
in place of $T'$.

Therefore, from \eqref{f5} we have
\begin{equation}\label{f9}
\widehat{T}\, \subset\, Z_G(S_\Theta)\, .
\end{equation}

{}From the maximality of $B$ it follows that $\widehat{T}$ is a maximal torus 
of $Z_G(S_\Theta)$. Consequently,
$$
Z(Z_G(S_\Theta)) \, \subset\, \widehat{T}\, .
$$

Let $$T \, \subset\, [Z_G(S_\Theta)\, , Z_G(S_\Theta)]$$ be a maximal 
$k$--torus, and let $A\, \subset\, Z(Z_G(S_\Theta))^0$ be a 
$k$--anisotropic torus (see \eqref{A}), such that
$$Z(Z_G(S_\Theta))^0 \,=\, AS_\Theta
~\, ~\text{ and } ~\,~ \widehat{T}\,=\, T A S_\Theta\, .$$

We will prove that the torus $T$ is $k$--anisotropic.

Let $T_1\, \subset\, T$ be the $k$--split part of $T$. From the maximality 
of $B$ it follows there is no nontrivial unipotent $k$--subgroup of $G$
such that
$$U\, \subset\, [Z_G(S_\Theta)\, , Z_G(S_\Theta)] ~\, ~\text{ and }\, ~
T\, \subset\, N_{[Z_G(S_\Theta)\, , Z_G(S_\Theta)]} (U)\, .
$$
Indeed, otherwise $B$ is strictly contained in $\widehat{T} UR_u(P_\Theta)$, 
contradicting the maximality of $B$. Now from Lemma \ref{lem3}, and the fact 
that $[Z_G(S_\Theta)\, , Z_G(S_\Theta)]$ is semisimple, we conclude that 
$T_1\,=\, \{e\}$. Thus $T$ is $k$--anisotropic.

Since $T$ is $k$--anisotropic, in view of \eqref{f10} and \eqref{f9}, we
conclude that $B\,=\, B_{\Theta, T}$.
\end{proof}

Theorem \ref{thm1} and Theorem \ref{thm2} together describe
the maximal connected solvable subgroups of $G$ defined over $k$.
It remains to give a criterion for two subgroups as in Theorem \ref{thm1} 
to be conjugate by some element of $G(k)$.

Let
$$
\Theta_1\, , \Theta_2 \,\subset\, \Phi
$$
be two admissible subsets. As before, we construct 
subgroups
$$
B_{\Theta_i, T_i} \,:=\, T_i Z(Z_G(S_{\Theta_i})) R_u (P_{\Theta_i})\, ,~\,
~ i\,=\, 1\, ,2\, ,
$$
where $T_i$ is a maximal $k$-torus of
$[Z_G(S_{\Theta_i})\, , Z_G(S_{\Theta_i})]$ which is $k$--anisotropic.

\begin{proposition}\label{thm3}
The two subgroups $B_{\Theta_1, T_1}$ and 
$B_{\Theta_2,T_2}$ are conjugate by 
some element of $G(k)$ if and only if $$\Theta_1\,= \,\Theta_2~\,~ \text{ and 
}~\,~ c \,T_1\, c^{-1} \,= \,T_2$$ for some $c\,\in\, P_{\Theta_1}(k) 
\,=\,P_{\Theta_2} (k)$ satisfying the condition that $c \, Z_G ( 
S_{\Theta_1}) \, c^{-1} \,=\, Z_G (S_{\Theta_1})$.
\end{proposition}

\begin{proof} 
First assume that $$\Theta_1 \,= \,\Theta_2~\,~
\text{ and }~\,~ c \,T_1\, c^{-1} \,= \,T_2\, ,$$
where $c\,\in\, P_{\Theta_1}(k)$
with $c \, Z_G ( S_{\Theta_1}) \, c^{-1} \,=\, Z_G (
S_{\Theta_1})$. Let $\Theta \,:=\,\Theta_1 \,= \,\Theta_2$.
As $c$ normalizes $R_u (P_{\Theta})$ and $Z_G (S_{\Theta})$,
we have
$$
c \, T_1 Z(Z_G(S_{\Theta_1})) R_u (P_{\Theta_1}) \, c^{-1}\,=\,
T_2 Z(Z_G(S_{\Theta_2})) R_u (P_{\Theta_2})\, .
$$
In particular, $B_{\Theta_1,T_1}$ and $B_{\Theta_2,T_2}$
are $G(k)$--conjugate.

We will now prove the converse. Assume that there is an element $c \,\in\, 
G(k)$ such that $$c \,B_{\Theta_1,T_1} \,c^{-1}\,=\,B_{\Theta_2,T_2}\, .$$ 
Then we have
$$
c\, R_u (P_{\Theta_1})\, c^{-1}\, = \,R_u (P_{\Theta_2})
~\, \text{ and }\, ~
c\, T_1 Z(Z_G(S_{\Theta_1}))\, c^{-1}\, 
=\, \beta\, T_2 Z(Z_G( S_{\Theta_2}))\, \beta^{-1}
$$
for some $\beta \,\in\, R_u (P_{\Theta_2}) (k)$.
Thus, without loss of generality, we may, and we will, assume that
\begin{equation}\label{x1}
c\, R_u (P_{\Theta_1})\, c^{-1}\, = \,R_u (P_{\Theta_2})
~\, \text{ and }\, ~
c\, T_1 Z(Z_G( S_{\Theta_1})) \, c^{-1}\, =\,T_2 Z(Z_G( S_{\Theta_2})). 
\end{equation}
But as $S_{\Theta_i}$ is the $k$--split part of the torus $T_i Z(Z_G( 
S_{\Theta_i}))$, it follows that 
$$
c\, S_{\Theta_1}\, c^{-1}\, =\, S_{\Theta_2}\, .
$$
In particular,
\begin{equation}\label{g5}
c\, Z_G(S_{\Theta_1})\, c^{-1}\, =\, Z_G(S_{\Theta_2})\, .
\end{equation}
Thus
\begin{equation}\label{f4}
c\, P_{\Theta_1}\, c^{-1}\, =\, c\, Z_G(S_{\Theta_1}) R_u (P_{\Theta_1})\, 
c^{-1}\, =\, Z_G(S_{\Theta_2}) R_u (P_{\Theta_2})\,=\,P_{\Theta_2}\, .
\end{equation}
Using \cite[p. 234, Proposition 21.12]{Bo} it follows 
that $\Theta_1 \,=\, \Theta_2$.

As before, we set $\Theta\,:=\,\Theta_1\,=\,\Theta_2$. Then from \eqref{f4} it 
follows that $c \, P_{\Theta} \, c^{-1} \,= \,P_{\Theta}$. This implies that 
$c \,\in\, P_{\Theta}(k)$. Again from \eqref{g5} we have that $c \, 
Z_G(S_{\Theta})\, c^{-1} =\, Z_G(S_{\Theta})$. Now, as $T_i$ is a maximal 
$k$--torus in $[Z_G (S_\Theta)\, ,Z_G (S_\Theta)]$, it follows that
$$T_i \,=\, (T_i Z(Z_G (S_{\Theta} )) \cap [Z_G (S_{\Theta} )\, ,Z_G 
(S_{\Theta} )])^0\, .$$
Using \eqref{x1} it follows that $c\, T_1 Z(Z_G( S_{\Theta})) \, c^{-1}\, =\, 
T_2 Z(Z_G(S_{\Theta}))$. Thus $c \, T_1 c^{-1} \,=\, T_2$. In view of 
\eqref{g5}, the proof is now complete.
\end{proof}

\section{A criterion to be in a maximal solvable subgroup}

As before, let $G$ be a semisimple affine algebraic group defined over a field 
$k$ of 
characteristic zero. If $g \,\in\, G(k)$ is semisimple or unipotent, then it is 
easy to see that $g$ lies in a connected abelian $k$--subgroup of $G$. A
connected abelian $k$--subgroup of $G$ clearly
lies in a maximal connected solvable $k$-subgroup of $G$.

However, if $g \,\in\, G(k)$ is arbitrary then it is not true in general that 
$g$ is contained in a connected abelian $k$-subgroup of $G$.
When $k \,=\, {\mathbb R}$, in \cite[Theorem 5.11]{Ch}, a sufficient condition 
is given for an element $g\, \in\, G(k)$ to lie in a $G(k)$--conjugate of 
$B_{\Theta, T}$. 
We point out that the conclusion of \cite[Theorem 5.11]{Ch} holds
for a general $k$ of characteristic zero. We have nothing new to add here other 
than noting that the transition of \cite[Section 5]{Ch} from the case of 
$k\,=\, {\mathbb R}$ to the case of a general $k$ of characteristic zero goes 
through without any difficulty.

Let $W_k$ denote the $k$--\textit{Weyl group} $N_G (S)/ Z_G (S)$.
We note that $W_k$ acts on $S$ and, in particular, induces an action on
the power set (= set of subsets) of $\Delta$.

\begin{lemma}\label{last}
Let $s$ be a semisimple element in $G(k)$. Then any maximal $k$--split torus of 
$Z_G (s)^0$ is $G(k)$--conjugate to a standard $k$--split torus 
$S_\Theta$ for some $\Theta \,\subset\, \Phi$. Moreover, if $\Theta' 
\,\subset\, \Phi$, then some $G(k)$--conjugate of the standard $k$--split torus 
$S_{\Theta'}$ is a maximal $k$--split torus of $Z_G (s)^0$ if and only if 
$\Theta$ and $\Theta'$ are $W_k(k)$--conjugate.
\end{lemma} 

\begin{proof}
This lemma is proved in \cite{Ch} under the assumption that $k \,=\, {\mathbb 
R}$ (see \cite[Corollary 5.7]{Ch}). The proof of the first part of the lemma is 
exactly identical to the proof of the first part in \cite[Corollary 5.7]{Ch}; 
we just need to replace ${\mathbb R}$ by $k$.

For the proof of the second part, we need a bit more justification than that
is given in \cite[Corollary 5.7]{Ch}. Let $\Theta\, , \Theta'\,\subset\,\Phi$ 
be such that some $G(k)$--conjugates of both $S_{\Theta}$ and $S_{\Theta'}$ are 
maximal $k$--split tori of $Z_G (s)^0$. Then by conjugacy of maximal $k$--split 
tori it follows that there is an element $c\,\in\, Z_G (s)^0 (k)$ such that 
$$
c S_{\Theta} c^{-1} \,=\, S_{\Theta'}\, . 
$$
As $S_{\Theta}\, ,S_{\Theta'}$ are both subtori of $S$, using \cite[Corollary 
4.22]{BT2}, we conclude that there is an element $a\,\in\, N_G (S) (k)$ such 
that 
$$
a x a^{-1} \,=\, cx c^{-1}, ~ \,\, 
\forall x \,\in \, S_{\Theta}\, .
$$
In particular, $\Theta$ and $\Theta'$ are $W_k (k)$--conjugate.
\end{proof}

Take any $s\, \in\, G(k)$.
Let $\Theta \,\subset \,\Phi$ be such that some $G(k)$--conjugate 
of the $k$--split torus $S_\Theta$ is a maximal $k$--split torus of
$Z_G (s)^0$. Then $s$ is said to be of \textit{type} $\Theta$; see
\cite[Definition 5.8]{Ch}. If $s$ is of type $\Theta$, then note that $\Theta$ 
is necessarily an admissible subset of $\Phi$.

Assume that $s$ is of type $\Theta$. Take
any $\Theta' \,\subset \,\Phi$. Then from Lemma \ref{last} it follows that
$s$ is of type $\Theta'$ if and only if $\Theta$ and $\Theta'$ are $W_k 
(k)$--conjugate.

\begin{theorem}\label{C}
Take any $g \,\in\, G(k)$. Let $g_s$ be the semisimple part of $g$.
Assume that $g_s$ is of type $\Theta$ (in particular,
$\Theta$ is admissible). Then there is an element $c \,\in\, G(k)$, and a 
maximal $k$--torus $T$ of $[Z_G (S_\Theta)\, , Z_G (S_\Theta)]$ which 
is $k$-anisotropic, such that
$$
cgc^{-1} \, \in\, B_{\Theta,T}\, .
$$
In particular, $G(k)$ is the union of all $G(k)$--conjugates of $B_{\Theta,T} 
(k)$ and all maximal $k$--tori of $[Z_G (S_\Theta), Z_G (S_\Theta)]$ which are 
$k$--anisotropic, where $\Theta$ runs over all admissible subsets of $\Phi$.
\end{theorem}

Theorem \ref{C} was proved in \cite{Ch} for $k \,=\, {\mathbb R}$ (see 
\cite[Theorem 5.11]{Ch}). The proof of Theorem 5.11 in \cite{Ch} works for 
any $k$ of characteristic zero. Hence we omit the proof.

\medskip
\noindent \textbf{Acknowledgements.}\,
The first-named author thanks KFUPM for funding Research Project IN101026. The 
third-named author thanks School of Mathematics, Tata Institute of Fundamental Research for its 
hospitality.


\begin{thebibliography}{AAA}

\bibitem[BT1]{BT} A. Borel and J. Tits, \'El\'ements unipotents et 
sous-groupes paraboliques de groupes r\'eductifs. I, \textit{Invent.
Math.} \textbf{12} (1971), 95--104.

\bibitem[BT2]{BT2} A. Borel and J. Tits, Groupes r\'eductifs, \textit{Inst. 
Hautes \'Etudes Sci. Publ. Math.} \textbf{27} (1965), 55--150.

\bibitem[Bo]{Bo} A. Borel, \textit{Linear algebraic groups},
Graduate Texts in Mathematics, 126. Springer-Verlag, New York, 1991.

\bibitem[Ch]{Ch} P. Chatterjee, Surjectivity of power maps of real
algebraic groups, \textit{Adv. Math.} \textbf{226} (2011), 4639--4666.

\bibitem[DT]{DT} D.Z. Djokovi\'c and N.Q. Thang, Conjugacy classes of maximal tori 
in semisimple real algebraic groups and applications, \textit{Can. Jour. Math.} 
\textbf{46} (1994), 699--717.

\bibitem[Ma]{Ma} H. Matsumoto, Quelques remarques sur les groupes de Lie 
alg\'ebriques r\'eels, \textit{Jour. Math. Soc. Japan} \textbf{16} (1964), 
419--446.

\bibitem[Mo]{Mo} G.D. Mostow, On maximal subgroups of real Lie groups,
\textit{Ann. of Math.} \textbf{73} (1961), 20--48.

\bibitem[Pl]{P} V.P. Platonov, Proof of the finiteness hypothesis for
the solvable subgroups of algebraic groups,
\textit{Sibirsk. Mat. Z.} \textbf{10} (1969), 1084--1090.


\end{thebibliography}
\end{document}